\documentclass[review,3p,authoryear,square]{elsarticle}
\usepackage{amssymb}
\usepackage{amsmath}
\usepackage{amsthm}
\usepackage{ushort}
\usepackage{enumitem}
\usepackage{hyperref}
\hypersetup{%
  pdfauthor={Lishun XIAO},%
  pdfstartview=FitH,%
  CJKbookmarks=true,%
  bookmarksnumbered=true,%
  bookmarksopen=true,%
  colorlinks=true, linkcolor=blue, urlcolor=blue, citecolor=blue, %
}
\newcommand{\rtn}{\mathrm{\mathbf{R}}}

\newcommand*{\PR}{\mathrm{\mathbf{P}}}
\newcommand*{\EX}{\mathrm{\mathbf{E}}}
\newcommand*{\dif}{\,\mathrm{d}}
\newcommand*{\calF}{\mathcal{F}}
\newcommand*{\calL}{\mathcal{L}}
\newcommand{\calS}{\mathcal{S}}

\newcommand{\calH}{\mathcal{H}}
\newcommand{\calK}{\mathcal{K}}
\newcommand*{\prs}{\PR\text{\,--\,\,}a.s.}
\newcommand*{\ptss}{\dif\PR\!\times\!\!\dif s\text{\,--\,\,}a.e.}
\newcommand*{\pts}{\dif\PR\!\times\!\!\dif t\text{\,--\,\,}a.e.}
\newcommand*{\dte}{\dif t\text{\,--\,\,}a.e.}
\newcommand{\tT}{[0,T]}
\newcommand{\intT}[2][T]{\int^{#1}_{#2}}
\newcommand{\me}{\mathrm{e}}
\newcommand{\one}[1]{{\bf 1}_{#1}}
\newcommand{\EXlr}[1]{\mathrm{\mathbf{E}}\left[#1\right]}
\newcommand*{\vp}{\varepsilon}
\newcommand{\tvptau}{{(t+\vp)\wedge\tau}}
\newcommand{\intTe}[1][t]{\intT[t+\vp]{#1}}
\newcommand{\Me}[1][t]{M^{\vp}_{#1}}
\newcommand{\Ne}[1][t]{N^{\vp}_{#1}}

\usepackage{cleveref}
\crefname{thm}{Theorem}{Theorems}
\crefname{pro}{Proposition}{Propositions}
\crefname{lem}{Lemma}{Lemmas}
\crefname{rmk}{Remark}{Remarks}
\crefname{cor}{Corollary}{Corollaries}
\crefname{dfn}{Definition}{Definitions}
\crefname{ex}{Example}{Examples}
\crefname{section}{Section}{Sections}
\crefname{subsection}{Subsection}{Subsections}

\newtheorem{thm}{Theorem}
\newtheorem{lem}[thm]{Lemma}
\newtheorem{pro}[thm]{Proposition}
\newtheorem{rmk}[thm]{Remark}
\newtheorem{dfn}[thm]{Definition}
\newtheorem{cor}[thm]{Corollary}
\newtheorem{ex}[thm]{Example}

\makeatletter
\def\ps@pprintTitle{%
	\let\@oddhead\@empty
	\let\@evenhead\@empty
	\def\@oddfoot{\footnotesize\itshape
		\hfill\today}%
	\let\@evenfoot\@oddfoot}
\makeatother
\begin{document}
\begin{frontmatter}


\title{{\boldmath\bf A representation theorem for generators of BSDEs with general growth generators in $y$ and its applications}\tnoteref{found}}
\tnotetext[found]{This work was supported by the National Natural Science Foundation of China (Nos.\;11371362 and 11601509), the  Natural Science Foundation of Jiangsu Province (No.\;BK20150167) and the Research Innovation Program for College Graduates of Jiangsu Province (No.\;KYZZ15\_0376).}
\author{Lishun XIAO}
\author{Shengjun FAN\corref{cor1}}%
\ead{f\_s\_j@126.com}
\cortext[cor1]{Corresponding author}

\address{School of Mathematics, China University of Mining and Technology, Xuzhou, Jiangsu, 221116, P.R. China}

\begin{abstract}
  In this paper we first prove a general representation theorem for generators of backward stochastic differential equations (BSDEs for short) by utilizing a localization method involved with stopping time tools and approximation techniques, where the generators only need to satisfy a weak monotonicity condition and a general growth condition in $y$ and a Lipschitz condition in $z$. This result basically solves the problem of representation theorems for generators of BSDEs with general growth generators in $y$. Then, such representation theorem is adopted to prove a probabilistic formula, in viscosity sense, of semilinear parabolic PDEs of second order. The representation theorem approach seems to be a potential tool to the research of viscosity solutions of PDEs. 
\end{abstract}

\end{frontmatter}

\section{Introduction}
\label{sec:Introduction}
By \citet*{PardouxPeng1990SCL} we know that the following one dimensional backward stochastic differential equation (BSDE for short): 
\begin{equation}\label{eq:BSDEsOneDimension}
  Y_t=\xi+\intT{t}{g(s,Y_s,Z_s)}\dif s
  -\intT{t}\langle Z_s,\dif B_s\rangle,\quad t\in\tT,
\end{equation}
admits a unique adapted and square-integrable solution $(Y_t(g,T,\xi),Z_t(g,T,\xi))_{t\in\tT}$,  provided that $g$ is Lipschitz continuous with respect to $(y,z)$, and that $\xi$ and $(g(t,0,0))_{t\in\tT}$ are square-integrable. We call $T$ with $0\leq T<\infty$ the terminal time, $\xi$ the terminal condition and $g$ the generator. BSDE \eqref{eq:BSDEsOneDimension} is also denoted by BSDE $(g,T,\xi)$. 

\citet*{BriandCoquetHuMeminPeng2000ECIP} constructed a representation theorem for the generator of BSDE \eqref{eq:BSDEsOneDimension}: for each $(y,z)\in\rtn\times\rtn^d$,
\begin{equation}\label{eq:RepresentationEquation}
  g(t,y,z)=L^2-\lim_{\vp\to 0^+}\frac{1}{\vp}
  \left[
    Y_t(g,t+\vp,y+z\cdot(B_{t+\vp}-B_t))-y
  \right], \quad \forall t\in[0,T),
\end{equation}
where the generator $g$ is Lipschitz continuous in $(y,z)$ and satisfies two additional conditions, i.e., $\EX\big[\sup_{t\in\tT}|g(t,0,0)|^2\big]<\infty$ and $g(t,y,z)$ is continuous with respect to $t$. It has been widely recognized that the representation theorems provide a crucial tool for investigating properties of the generator and $g$-expectations by virtue of solutions of the BSDEs. Note that $g$-expectations were first introduced by \citet*{Peng1997BSDEP} to explain the nonlinear phenomena in mathematical finance. More details on this direction are referred to \citet*{Jiang2005cSPA}, \citet*{Jiang2008TAAP},  \citet*{Jia2010SPA} and \citet*{MaYao2010SAA}, etc.

Particularly, many attempts have been made to weaken the conditions required by the generator $g$. For instance,  \citet*{Jiang2006ScienceInChina} and \citet*{Jiang2008TAAP} eliminated the above two additional assumptions and proved that \eqref{eq:RepresentationEquation} remains valid for $\dte$ $t\in\tT$ in $L^p$ $(1\leq p<2)$ sense;   
\citet*{Jia2010SPA} proved a representation theorem where the generator $g$ is Lipschitz continuous in $y$ and uniformly continuous in $z$;  \citet*{FanJiang2010JCAM} presented a representation theorem where $g$ is only continuous and of linear growth in $(y,z)$. However, all these works only handled the case of the generator $g$ with a linear growth in $(y,z)$. In subsequent works on this direction, researchers were devoted to the case that the generator $g$ has a nonlinear growth in $(y,z)$. For example, \citet*{MaYao2010SAA} investigated the representation theorem when the generator $g$ has a quadratic growth in $z$. They applied a stopping time technique to truncate the terminal condition $|B_s-B_t|$, which ensures the first part of the solution is essentially bounded. This idea will be borrowed to deal with our problems. Furthermore, by adopting an approximation technique, \citet*{FanJiangXu2011EJP} weakened the conditions required by the generator $g$ in $y$ to a monotonicity condition with a polynomial growth. Recently, \citet*{ZhengLi2015SPL} further studied a representation theorem when the generator $g$ is monotonic with a convex growth in $y$, and has a quadratic growth in $z$, where a vital property of convex functions, superadditivity, plays a key role. We mention that up to now the problem of representation theorems for generators of BSDEs with general growth generators in $y$ has not been solved completely. 

In the present paper, we first prove a general representation theorem for generators of BSDEs where the generator $g$ only needs to satisfy a weak monotonicity condition and a general growth condition in $y$, and a Lipschitz condition in $z$. This basically answers the problem of representation theorems for generators of BSDEs with general growth generators in $y$, which has been standing for more than decade. Then, utilizing such representation theorem we establish a converse comparison theorem for solutions of BSDEs and a probabilistic formula for viscosity solutions of systems of second order semilinear parabolic partial differential equations (PDEs for short). 

For the representation theorem, the difficulty mainly comes from the general growth conditions of the generator $g$ in $y$. The whole idea to handle it is that we first use a stopping time to truncate the terminal condition $|B_s-B_t|$ and $\int^s_t|g(r,0,0)|^2\dif r$ such that the first part of the solution, $(Y_t)_{t\in\tT}$, is essentially bounded by a constant $K>0$, and then treat $g(t,q_K(y),z)$ with $q_K(y)=Ky/(|y|\vee K)$, which has a linear growth in $y$, instead of $g(t,y,z)$, by virtue of an approximation technique. We also mention that the proof procedure of our representation theorem is both simpler and clearer than those in the literature.  

For the probabilistic formula for solutions of semilinear parabolic PDEs, which is actually a nonlinear Feynman-Kac formula, its connection with the representation theorem for generators of BSDEs is discovered for the first time. More precisely, the probabilistic formula for semilinear PDEs holds as soon as such representation theorem holds. This can also be regarded as a new application of the representation theorem. We believe that the representation theorem will be a potential tool to investigate the viscosity solutions of PDEs. 

The rest of this paper is organized as follows. \cref{sec:Preliminaries} contains some notations and preliminaries. \cref{sec:RepresentationTheorem} presents the representation theorem and a converse comparison theorem. \cref{sec:Applications} proves the probabilistic formula for viscosity solutions of second order semilinear parabolic PDEs.

\section{Preliminaries}
\label{sec:Preliminaries}
Let $T>0$ be a given real number, $(\Omega,\calF,\PR)$ a probability space carrying a standard $d$-dimensional Brownian motion $(B_t)_{t\geq 0}$ and $(\calF_t)_{t\geq 0}$ the natural $\sigma$-algebra filtration generated by $(B_t)_{t\geq 0}$. We assume that $\calF_T=\calF$ and $(\calF_t)_{t\geq 0}$ is right-continuous and complete. We  denote the Euclidean norm and dot product by $|\cdot|$ and $\langle\cdot,\cdot\rangle$, respectively. For each real number $p>1$, let $L^p(\calF_t;\rtn)$ be the set of real-valued and $(\calF_t)$-measurable random variables $\xi$ such that $\EX[|\xi|^p]<\infty$; and  $\calS^2(0,T;\rtn)$ (or $\calS^2$ for simplicity) the set of real-valued, $(\calF_t)$-adapted and continuous processes $(Y_t)_{t\in\tT}$ such that $\EX\big[\sup_{t\in\tT}|Y_t|^2\big]<\infty$.
Moreover, let $\calH^2(0,T;\rtn^{d})$ (or $\calH^2$) denote the set of (equivalent classes of) $(\calF_t)$-progressively measurable ${\rtn}^{d}$-valued processes $(Z_t)_{t\in\tT}$ such that $\EX\big[\int_0^T |Z_t|^2\dif t\big]<\infty$. Finally, let $\calK$ denote the set of nondecreasing and concave continuous function  $\kappa:\rtn^+\mapsto\rtn^+$ satisfying $\kappa(0)=0$, $\kappa(u)>0$ for each $u>0$ and $\int_{0^+}\dif u/\kappa(u)=\infty$. Since $\kappa(\cdot)$ increases at most linearly, we denote the linear growth constant of $\kappa\in\calK$ by $A>0$, i.e., $\kappa(u)\leq A+Au$ for all $u\in\rtn^+$. 

We will deal with the one dimensional BSDE which is an equation of type \eqref{eq:BSDEsOneDimension}, where the terminal condition $\xi$ is $\calF_T$-measurable and the generator $g(t,y,z):\Omega\times\tT\times\rtn\times\rtn^d\mapsto\rtn$ is $(\calF_t)$-progressively measurable for each $(y,z)$, where and hereafter we suppress $\omega$ for brevity. In this paper, a solution of BSDE \eqref{eq:BSDEsOneDimension} is a pair of processes $\big(Y_t(g,T,\xi),Z_t(g,T,\xi)\big)_{t\in\tT}$ in $\calS^2\times\calH^2$ which satisfies BSDE \eqref{eq:BSDEsOneDimension}. Next we list some lemmas which will be used in the proof of our main result.


\begin{lem}[Proposition 2.2 in \citet*{Jiang2008TAAP}]\label{lem:ExtendLebesgueLemma}
  Let $1\leq p<2$. For any $(\phi_t)_{t\in\tT}\in\calH^2(0,T;\rtn)$, we have
  \begin{equation*}
    \phi_t=L^p-\lim_{\vp\to0^+}\frac{1}{\vp}\intTe\phi_s\dif s,\quad \dte\ t\in[0,T).
  \end{equation*}
\end{lem}

The following lemma presents an a priori estimate for solutions of BSDE \eqref{eq:BSDEsOneDimension}, which comes from Proposition 3.1 in \citet*{FanJiang2013AMSE}. To state it, we need the following assumption. 
\begin{enumerate}
  \renewcommand{\theenumi}{(A)}
  \renewcommand{\labelenumi}{\theenumi}
  \item\label{A:EstimateAssumption} There exists a positive constant $\mu$ such that $\pts$, for each $y\in\rtn$ and $z\in\rtn^d$,
  \[yg(t,y,z)\leq \kappa(|y|^2)+\mu|y||z|+|y|f_t,\]
  where $\kappa\in\calK$ and $(f_t)_{t\in\tT}$ is a non-negative, real-valued and $(\calF_t)$-progressively measurable process with $\EX\big[\big(\intT{0}f_t\dif t\big)^2\big]<\infty$.
\end{enumerate}

\begin{lem}\label{lem:EstimateForSolutions}
  Assume that $g$ satisfies \ref{A:EstimateAssumption} and $(Y_t,Z_t)_{t\in\tT}$ is a solution of BSDE \eqref{eq:BSDEsOneDimension}. Then there exists a constant $C>0$ depending on $\mu$ and $T$ such that for each $0\leq u\leq t\leq T$,
  \begin{align*}
	&\EX\left[\left.\sup_{r\in[t,T]}|Y_r|^2\right|\calF_u\right]
	+\EX\left[\left.\intT{t}|Z_r|^2\dif r\right|\calF_u\right]\\
	&\leq C
	\left(\EX\left[\left.|\xi|^2\right|\calF_u\right]
	+\intT{t}\kappa\left(\EX\big[|Y_r|^2|\calF_u\big]\right)\dif r +\EX\left[\left.\left(\intT{t}f_r\dif r \right)^2\right|\calF_u\right]
	\right).
  \end{align*}
\end{lem}

Next we introduce an existence and uniqueness result for solutions of BSDE \eqref{eq:BSDEsOneDimension}, which is actually a corollary of Theorem 1 in \citet*{Fan2015JMAA}. The necessary assumptions are given as follows.
\begin{enumerate}
	\renewcommand{\theenumi}{(H\arabic{enumi})}
	\renewcommand{\labelenumi}{\theenumi}
	\item\label{H:Integrable}
	$\EX\Big[\intT{0}|g(t,0,0)|^2\dif t\Big]<\infty$;
	\item\label{H:gContinuousInY}
	$\pts$, for each $z\in\rtn^d$, $y\mapsto g(t,y,z)$ is continuous;
	\item\label{H:gWeakMonotonicityInY} $g$ satisfies a weak monotonicity condition in $y$, i.e., there exists a function $\rho\in\calK$ such that $\pts$, for each $y_1$, $y_2\in\rtn$ and $z\in\rtn^{d}$,
	\[
	(y_1-y_2)\big(g(t,y_1,z)-g(t,y_2,z)\big)\leq \rho(|y_1-y_2|^2);
	\]
	\item\label{H:gGeneralizedGeneralGrowthInY}
	$g$ has a general growth in $y$, i.e., $\psi_{\alpha}(t)\!:=\!\sup_{|y|\leq \alpha}|g(t,y,0)\!-\!g(t,0,0)|\!\in\! \calH^2(0,T;\rtn)$ for each $\alpha\geq0$;
	\item\label{H:gLipschitzInZ} There exists a constant $\lambda\geq 0$ such that $\pts$, for each $y\in\rtn$ and $z_1$, $z_2\in\rtn^d$,
	\[
	|g(t,y,z_1)-g(t,y,z_2)|\leq \lambda|z_1-z_2|.
	\]
\end{enumerate}

We also need the following growth condition introduced by \citet*{Pardoux1999NADEC} to compare with our result.

\begin{enumerate}
	\renewcommand{\theenumi}{(H\arabic{enumi}')}
	\renewcommand{\labelenumi}{\theenumi}
	\setcounter{enumi}{3}
	\item\label{H:ArbitraryGrowth}
  There exists a continuous increasing function $\varphi:\rtn^+\mapsto\rtn^+$ such that $\pts$, for each $y\in\rtn$, $|g(t,y,0)|\leq|g(t,0,0)|+\varphi(|y|)$.
\end{enumerate}

\begin{rmk}
  (i) It is clear that \ref{H:gGeneralizedGeneralGrowthInY} is weaker than \ref{H:ArbitraryGrowth} under the condition \ref{H:Integrable}. If the increasing function $\varphi(\cdot)$ is also  convex, \ref{H:ArbitraryGrowth} becomes the convex growth condition adopted by \citet*{ZhengLi2015SPL}. Hence, the general growth condition \ref{H:gGeneralizedGeneralGrowthInY} is also weaker than the convex growth condition. We mention that the supperadditivity of convex function $\varphi(\cdot)$ plays a key role in the proof of \citet*{ZhengLi2015SPL}. Yet $\psi_\alpha(t)$ in \ref{H:gGeneralizedGeneralGrowthInY} enjoys no other more explicit expressions except for the integrable property. This is the main difficulty to prove the corresponding representation theorem for generators under \ref{H:gGeneralizedGeneralGrowthInY}.
  
  (ii) Originally, it is required that $\psi_\alpha(t)$ in \ref{H:gGeneralizedGeneralGrowthInY} belongs to $L^1(\tT\times\Omega)$ in the existence and uniqueness theorem of the solution established in \citet*{BriandDelyonHu2003SPA}. However, to obtain our representation theorem the integrability of $\psi_\alpha(t)$ needs to be strengthened to the case in \ref{H:gGeneralizedGeneralGrowthInY}.
\end{rmk}

\begin{pro}\label{pro:ExistenceUniquenessOfSolutions}
  Let $g$ satisfy \ref{H:Integrable} -- \ref{H:gLipschitzInZ}. Then for each $\xi\in L^2(\calF_T;\rtn)$, BSDE \eqref{eq:BSDEsOneDimension} admits a unique solution $\big(Y_t(g,T,\xi),Z_t(g,T,\xi)\big)_{t\in\tT}$ in $\calS^2\times\calH^2$.
\end{pro}

The following \cref{pro:SequenceOfGenerator} can be seen as the starting point of this paper. The crucial steps to prove \cref{pro:SequenceOfGenerator} are the approximation using the convolution technique proposed by \citet*{LepeltierSanMartin1997SPL}. For readers' convenience, the proof procedures are given to adapt our settings. 
\begin{pro}\label{pro:SequenceOfGenerator}
  Assume that the generator $g$ satisfies \ref{H:Integrable}, \ref{H:gContinuousInY} and \ref{H:gGeneralizedGeneralGrowthInY}. Then for each $\alpha\geq 0$, there exists a $(\calF_t)$-progressively measurable process sequence $\{(g^n_\alpha(t))_{t\in\tT}\}^\infty_{n=1}$ such that $\lim_{n\to\infty}\EX[|g^n_\alpha(t)|^2]=0$ for $\dte$ $t\in\tT$, and $\pts$, for each $n\geq 1$ and $y\in\rtn$, we have
  \begin{equation*}
    |g^n_\alpha(t)|\leq 2\psi_\alpha(t)+4|g(t,0,0)|\quad\text{and}\quad
    |g(t,q_\alpha(y),0)-g(t,0,0)|
    \leq|g^n_\alpha(t)|+2n|y|,
  \end{equation*}
  where and hereafter $q_\alpha(y):=\alpha y/(|y|\vee\alpha)$ for each $\alpha\geq 0$ and $y\in\rtn$.
\end{pro}
\begin{proof}
  Let the assumptions hold and fix a real number $\alpha\geq 0$. Note that $|q_\alpha(y)|\leq\alpha\wedge|y|$. By \ref{H:gGeneralizedGeneralGrowthInY} we know that $\pts$, for each $y\in\rtn$,
  \begin{equation}\label{eq:GtQAlphaLeqPsiPlusG}
    |g(t,q_\alpha(y),0)|\leq \psi_\alpha(t)+|g(t,0,0)|.
  \end{equation}
  Then, for each $n\geq 1$, we can define the following two $(\calF_t)$-progressively measurable processes:
  \begin{equation*}
    \ushort{g}^n_{\alpha}(t):=\inf_{u\in\rtn}\{g(t,q_\alpha(u),0)+n|u|\},\quad
    \bar g^n_\alpha(t):=\sup_{u\in\rtn}\{g(t,q_\alpha(u),0)-n|u|\}.
  \end{equation*}
  Now we show some properties of $\ushort{g}^n_\alpha(t)$. Firstly, it follows from \eqref{eq:GtQAlphaLeqPsiPlusG} that $\pts$, for each $u\in\rtn$,  $g(t,q_\alpha(u),0)\geq -\psi_\alpha(t)-|g(t,0,0)|$, which implies that $\ushort{g}^n_\alpha(t)$ is well defined and for each $n\geq 1$,
  \[\ushort{g}^n_\alpha(t)\geq \inf_{u\in\rtn}\{-\psi_\alpha(t)-|g(t,0,0)|+n|u|\}\geq -\psi_\alpha(t)-|g(t,0,0)|.\]
  On the other hand, it is clear from the definition that $\ushort{g}^n_\alpha(t)$ is nondecreasing in $n$ and $\pts$, $\ushort{g}^n_\alpha(t)\leq g(t,q_\alpha(0),0)=g(t,0,0)$. Thus, we know that $\pts$,  $\limsup_{n\to\infty}\ushort{g}^n_\alpha(t)\leq g(t,0,0)$ and for each $n\geq 1$,
  \begin{equation}\label{eq:GnAlphaLeqPsiPlusG}
    |\ushort{g}^n_\alpha(t)|\leq \psi_\alpha(t)+|g(t,0,0)|.
  \end{equation}
  Moreover, it follows from the definition of $\ushort{g}^n_\alpha(t)$ and \eqref{eq:GtQAlphaLeqPsiPlusG} that there exists a sequence $\{u_n\}^\infty_{n=1}$ such that $\pts$,
  \begin{equation*}
  \ushort{g}^n_\alpha(t)\geq g(t,q_\alpha(u_n),0)+n|u_n|-\frac{1}{n}
  \geq -\psi_\alpha(t)-|g(t,0,0)|+n|u_n|-\frac{1}{n},
  \end{equation*}
  which together with \eqref{eq:GnAlphaLeqPsiPlusG} implies that $u_n\to 0$ as $n\to\infty$, and that $\pts$, $\liminf_{n\to\infty}\ushort{g}^n_\alpha(t)\geq \lim_{n\to\infty}g(t,q_\alpha(u_n))=g(t,0,0)$. Therefore, $\pts$,  $\lim_{n\to\infty}\ushort{g}^n_\alpha(t)\uparrow=g(t,0,0)$. Besides, assumptions \ref{H:Integrable} and \ref{H:gGeneralizedGeneralGrowthInY} indicate that $\EX[|g(t,0,0)|^2]<\infty$ and $\EX[|\psi_\alpha(t)|^2]<\infty$ for $\dte$ $t\in\tT$. Thus, Lebesgue's dominated convergence theorem together with \eqref {eq:GnAlphaLeqPsiPlusG} yields that for $\dte$ $t\in\tT$,
  \begin{equation}\label{eq:EXGnAlphaMinusG}
    \lim_{n\to\infty} \EX[|\ushort g^n_\alpha(t)-g(t,0,0)|^2]=0.
  \end{equation}
  In the sequel, using a similar argument to that above we can prove that for $\dte$ $t\in\tT$,
  \begin{equation}\label{eq:PropertiesOfOverlineG}
    \prs,\quad|\overline g^n_\alpha(t)|\leq \psi_\alpha(t)+|g(t,0,0)|,\quad \forall n\geq 1\quad \text{and}\quad \lim_{n\to\infty} \EX[|\overline g^n_\alpha(t)-g(t,0,0)|^2]=0. 
  \end{equation}
  Finally, by the definitions of $\ushort g^n_\alpha(t)$ and $\overline g^n_\alpha(t)$ we easily get that $\pts$, for each $n\geq 1$ and $y\in\rtn$,
  \begin{equation}\label{eq:GLeqUnderlineGAndGGeq}
    g(t,q_\alpha(y),0)-g(t,0,0)\geq \ushort g^n_\alpha(t)-g(t,0,0)-n|y|,\quad
    g(t,q_\alpha(y),0)-g(t,0,0)\leq\bar g^n_\alpha(t)-g(t,0,0)+n|y|.
  \end{equation}  
  Hence, we can pick the process sequence $\{(g^n_\alpha(t))_{t\in\tT}\}^\infty_{n=1}$ as follows:
  \[g^n_\alpha(t):=|\ushort g^n_\alpha(t)-g(t,0,0)|+|\bar g^n_\alpha(t)-g(t,0,0)|,\]
  and the desired conclusion follows from \eqref{eq:GnAlphaLeqPsiPlusG} -- \eqref{eq:GLeqUnderlineGAndGGeq}.  
\end{proof}

\section{A general representation theorem for generators of BSDEs}
\label{sec:RepresentationTheorem}
In this section we present the representation theorem and its detailed proof. Our proof is partially motivated by \citet*{MaYao2010SAA} and \citet*{FanJiangXu2011EJP}. An example is also provided to illustrate the novelty of our representation theorem. At the end of this section, a converse comparison theorem for solutions of BSDEs is established.

\begin{thm}[Representation Theorem]
\label{thm:RepresentationTheoremForGenerator}
  Assume that the generator $g$ satisfies \ref{H:Integrable} -- \ref{H:gLipschitzInZ}. Then for each $y\in\rtn$, $z\in\rtn^d$, $1\leq p<2$ and $\dte$ $t\in[0,T)$,
  \begin{equation*}
    g(t,y,z)=L^p-\lim_{\vp\to 0^+}\frac{1}{\vp}
           \left[
             Y_t(g,(t+\vp)\wedge\tau,y+\langle z,B_\tvptau-B_t\rangle)-y
           \right],
  \end{equation*}
  where $\tau:=\inf\big\{s\geq t:|B_s-B_t|+\int^s_t|g(r,0,0)|^2\dif r>1\big\}\wedge T$.
\end{thm}

\begin{cor}\label{cor:RepresentationTheoremForGenerator}
Assume that the generator $g$ is deterministic and satisfies \ref{H:Integrable} -- \ref{H:gLipschitzInZ}, and $\tau$ is defined in \cref{thm:RepresentationTheoremForGenerator}. Then for each $y\in\rtn$, $z\in\rtn^d$ and $\dte$ $t\in[0,T)$,
  \begin{equation*}
    g(t,y,z)=\lim_{\vp\to 0^+}\frac{1}{\vp}
           \left[
             Y_t(g,(t+\vp)\wedge\tau,y+\langle z,B_\tvptau-B_t\rangle)-y
           \right].
  \end{equation*}
\end{cor}

\begin{rmk}\label{rmk:GContinuousInT}
  If for each $y\in\rtn$ and $z\in\rtn^d$, $t\mapsto g(t,y,z)$ is continuous, then the corresponding identities in \cref{thm:RepresentationTheoremForGenerator} and  \cref{cor:RepresentationTheoremForGenerator} remain valid for all $t\in[0,T)$.
\end{rmk}

Next we provide an example of a generator which satisfies the general growth condition \ref{H:gGeneralizedGeneralGrowthInY} but does not meet \ref{H:ArbitraryGrowth}. 
\begin{ex}
  Let $T>0$ be a fixed real number. For each $(\omega,t,y,z)\in\Omega\times\tT\times\rtn\times\rtn^d$, define the generator $g$ as follows:
  \[g(t,y,z):=-\me^{y|B_t|}+h(|y|)+|z|,\]
  where $h(x):=(-x\ln x)\one{0\leq x\leq\delta}+\big(h'(\delta-)(x-\delta) +h(\delta)\big)\one{x>\delta}$ with $\delta$ small enough. It is not very hard to verify that the generator $g$ fulfills assumptions \ref{H:Integrable} -- \ref{H:gLipschitzInZ} with $\rho(\cdot)=h(\cdot)$, and that it does not satisfy the growth condition \ref{H:ArbitraryGrowth}. Hence, the existing works including \citet*{Jiang2006ScienceInChina}, \citet*{Jia2010SPA}, \citet*{FanJiangXu2011EJP} and \citet*{ZhengLi2015SPL} can not imply the conclusion of \cref{thm:RepresentationTheoremForGenerator}. 
\end{ex}


\begin{proof}[{\normalfont\bfseries Proof of \cref{thm:RepresentationTheoremForGenerator}}]
  Assume that $g$ satisfies \ref{H:Integrable}, \ref{H:gContinuousInY}, \ref{H:gWeakMonotonicityInY} with $\rho(\cdot)$, \ref{H:gGeneralizedGeneralGrowthInY} and \ref{H:gLipschitzInZ} with $\lambda$. Fix $(t,y,z)\in[0,T)\times\rtn\times\rtn^d$ and choose $\vp>0$ such that $\vp\leq T-t$. We define the following stopping time, 
  \[
    \tau:=\inf\left\{s\geq t: |B_s-B_t|+\int^s_t|g(r,0,0)|^2\dif r>1\right\}\wedge T.
  \]
  \cref{pro:ExistenceUniquenessOfSolutions} indicates that the following BSDE
  \begin{equation*}
    Y^\vp_s=y+\langle z,B_\tvptau-B_t\rangle+
    \intT[t+\vp]{s}\one{r<\tau}g(r,Y^\vp_r,Z^\vp_r)\dif r-\intT[t+\vp]{s}\langle Z^\vp_r,\dif B_r\rangle, \quad s\in[t,t+\vp],
  \end{equation*}
  admits a unique solution in $\calS^2\times\calH^2$, 
  \[\big(Y_s(g,(t+\vp)\wedge\tau,y+\langle z,B_\tvptau-B_t\rangle),Z_s(g,(t+\vp)\wedge\tau,y+\langle z,B_\tvptau-B_t\rangle)\big)_{s\in[t,t+\vp]}.\]
  For notational simplicity, it is denoted by $(Y^\vp_s,Z^\vp_s)_{s\in[t,t+\vp]}$. Next we first show that $Y^\vp_\cdot$ is bounded. It follows from \ref{H:gWeakMonotonicityInY} and \ref{H:gLipschitzInZ} for $g$ that $\ptss$, for each $\widetilde y\in\rtn$ and $\widetilde z\in\rtn^d$, 
  \begin{align*}
  \widetilde y\cdot\one{s<\tau} g(s,\widetilde{y},\widetilde{z})
  \leq \rho(|\widetilde y|^2)+|\widetilde y|\one{s<\tau}|g(s,0,\widetilde z)|
  \leq \rho(|\widetilde y|^2)+\lambda|\widetilde y||\widetilde z|+|\widetilde y|\one{s<\tau}|g(s,0,0)|,
  \end{align*}
  which means that \ref{A:EstimateAssumption} is satisfied by $\one{s<\tau}g$ with $\kappa(\cdot)=\rho(\cdot)$, $\mu=\lambda$ and $f_s=\one{s<\tau}|g(s,0,0)|$. Then, according to  \cref{lem:EstimateForSolutions}, H\"older's inequality and the definition of $\tau$ we obtain that there exists a constant $C>0$ depending on $\lambda$ and $T$ such that for each $t\leq u\leq s\leq t+\vp$,
  \begin{align*}
  \EX\Big[\big|Y^{\vp}_s\big|^2\Big|\calF_u\Big]
  &\leq 2C(|y|^2+|z|^2)+C
  \intTe[s]\!\!\rho\left(\EX\Big[\big| Y^\vp_r\big|^2\Big|\calF_u\Big]\right)\dif r
  +CT\EXlr{\left.\intTe[s]\!\one{r<\tau}|g(r,0,0)|^2\dif r\right|\calF_u}.\\
  &\leq C
  \intTe[s]\rho\left(\EX\Big[\big| Y^\vp_r\big|^2\Big|\calF_u\Big]\right)\dif r +2C(|y|^2+|z|^2)+CT.
  \end{align*}
  Furthermore, since $\rho(x)\leq Ax+A$ for all $x\geq 0$, Gronwall's inequality yields that for each $t\leq u\leq s\leq t+\vp$,
  \begin{equation*}
  \EX\Big[\big|Y^{\vp}_s\big|^2\Big|\calF_u\Big]\leq \big[2C(|y|^2+|z|^2)+CT+CAT\big]\me^{CAT}=:M^2.
  \end{equation*}
  Then, substituting $u=s$ in the previous inequality yields that $\prs$, $|Y^\vp_s|\leq M$ for each $s\in[t,t+\vp]$.   
  
  In what follows, we set, for each $s\in[t,t+\vp]$, 
  \begin{equation*}
    \widetilde Y^\vp_s:=
    Y^\vp_s-y-\langle z,B_{s\wedge\tau}-B_t\rangle,\qquad
    \widetilde Z^\vp_s:=Z^\vp_s-\one{s<\tau}z.
  \end{equation*}
  Then, it is not hard to verify that $(\widetilde Y^\vp_s,\widetilde Z^\vp_s)_{s\in[t,t+\vp]}\in\calS^2\times\calH^2$ is a solution of the following BSDE:
  \begin{equation}\label{eq:BSDEWithWidetildeYZ}
    \widetilde{Y}^\vp_s=
    \intT[t+\vp]{s}\widetilde g(r,\widetilde{Y}^\vp_r,\widetilde{Z}^\vp_r)\dif r
    -\intT[t+\vp]{s}\langle\widetilde{Z}^\vp_r,\dif B_r\rangle, \quad s\in[t,t+\vp],
  \end{equation}
  where $\widetilde g(s,\widetilde y,\widetilde z):=\one{s<\tau}g(s,\widetilde y+y+\langle z,B_{s\wedge\tau}-B_t\rangle,\widetilde z+z)$ for each $\widetilde y\in\rtn$ and $\widetilde z\in\rtn^d$. And by the definition of $\tau$ we know that $\prs$, for each $s\in[t,t+\vp]$,
  \begin{equation}\label{eq:TildeYVpBoundedByK}
    |\widetilde Y^\vp_s|\leq |Y^\vp_s|+|y|+|z|
    \leq M+|y|+|z|=:K.
  \end{equation}
  From the definition of $\widetilde g$, it is straightforward to check that $\widetilde g$ fulfills \ref{H:gContinuousInY}, \ref{H:gWeakMonotonicityInY} and \ref{H:gLipschitzInZ} with $\lambda$. Moreover, by \ref{H:gGeneralizedGeneralGrowthInY} and \ref{H:gLipschitzInZ} for $g$ we have that $\ptss$, for each $\widetilde\alpha\geq 0$ and $\widetilde y\in\rtn$ with $|\widetilde{y}|\leq\widetilde\alpha$, 
  \begin{align}
  |\widetilde g(s,\widetilde y,0)-\widetilde g(s,0,0)|
  \leq{}& 2\lambda|z|+|g(s,\widetilde y+y+\langle z,B_{s\wedge\tau}-B_t\rangle,0)-g(s,0,0)|\nonumber\\
  &+|g(s,y+\langle z,B_{s\wedge\tau}-B_t\rangle,0)-g(s,0,0)|\nonumber\\
  \leq{}& 2\lambda|z|+\psi_{\widetilde\alpha+|y|+|z|}(s)+\psi_{|y|+|z|}(s).\label{eq:TildePsiAlphaBoundedByPsi}
  \end{align}
  Set $\widetilde\psi_{\widetilde\alpha}(s):=\sup_{|\widetilde y|\leq\widetilde\alpha}|\widetilde g(s,\widetilde y,0)-\widetilde g(s,0,0)|$. Then
  \begin{equation*}
    \EXlr{\intT{0}|\widetilde\psi_{\widetilde\alpha}(r)|^2\dif r}
    \leq 12\lambda^2|z|^2T +3\EXlr{\intT{0}\big(|\psi_{\widetilde\alpha+|y|+|z|}(r)|^2+|\psi_{|y|+|z|}(r)|^2\big)\dif r}<\infty,
  \end{equation*}
  which implies that $\widetilde g$ satisfies \ref{H:gGeneralizedGeneralGrowthInY} with $\widetilde \psi_{\widetilde\alpha}(t)$.
  Similar arguments as \eqref{eq:TildePsiAlphaBoundedByPsi} yield that $\ptss$,
  \begin{equation*}
  |\widetilde{g}(s,0,0)|
  \leq \lambda|z|+\psi_{|y|+|z|}(s)+\one{s<\tau}|g(s,0,0)|,
  \end{equation*} 
  from which it follows that
  \[
    \EXlr{\intT{0}|\widetilde g(s,0,0)|^2\dif s}
    \leq 3\lambda^2|z|^2T+3\EXlr{\intT{0}|\psi_{|y|+|z|}(s)|^2\dif s}+3\EXlr{\intT{0}|g(s,0,0)|^2\dif s}<\infty.
  \]
  Hence, \ref{H:Integrable} holds true for $\widetilde g$. Thus, we have shown that $\widetilde g$ satisfies \ref{H:Integrable} -- \ref{H:gLipschitzInZ}. Now we can demonstrate the following \cref{lem:TidleYZVpLimitZero}.
  \begin{lem}\label{lem:TidleYZVpLimitZero}
    For $\dte$ $t\in[0,T)$, we have that
    \begin{equation*}
      \lim_{\vp\to0^+}\frac{1}{\vp}
      \EXlr{\sup_{r\in[t,t+\vp]}|\widetilde{Y}^\vp_r|^2
      +\intT[t+\vp]{t}|\widetilde Z^\vp_r|^2\dif r}=0.
    \end{equation*}
  \end{lem}

\begin{proof}
  It follows from \ref{H:gGeneralizedGeneralGrowthInY} and \ref{H:gLipschitzInZ} for $\widetilde g$ together with \eqref{eq:TildeYVpBoundedByK} that $\ptss$, 
  \begin{align*}
    \widetilde Y^\vp_s\cdot\widetilde g(s,\widetilde Y^\vp_s,\widetilde Z^\vp_s)
    &\leq \lambda|\widetilde Y^\vp_s||\widetilde Z^\vp_s|+|\widetilde Y^\vp_s|\big(|\widetilde g(s,\widetilde Y^\vp_s,0)-\widetilde g(s,0,0)|+|\widetilde g(s,0,0)|\big)\\
    &\leq \lambda|\widetilde Y^\vp_s||\widetilde Z^\vp_s|+|\widetilde Y^\vp_s|\big(\widetilde \psi_K(s)+|\widetilde g(s,0,0)|\big),
  \end{align*}
  which indicates that \ref{A:EstimateAssumption} is fulfilled by $\widetilde g$ with $\kappa(\cdot)\equiv 0$, $\mu=\lambda$ and $f_s=\widetilde \psi_K(s)+|\widetilde g(s,0,0)|$. Then \cref{lem:EstimateForSolutions} and H\"older's inequality contribute to the existence of a constant $C'>0$ such that 
  \begin{align*}
    \frac{1}{\vp}\EXlr{\sup_{r\in[t,t+\vp]}|\widetilde Y^\vp_r|^2+\intTe|\widetilde Z^\vp_r|^2\dif t}
    &\leq \frac{C'}{\vp}\EXlr{\left(\intTe\big(\widetilde \psi_K(r)+|\widetilde g(r,0,0)|\big)\dif r\right)^2}\\
    &\leq 2C'\EXlr{\intTe|\widetilde \psi_K(r)|^2\dif r}+2C'\EXlr{\intTe|\widetilde g(r,0,0)|^2\dif r}.
  \end{align*}
  Thus, \ref{H:Integrable} and  \ref{H:gGeneralizedGeneralGrowthInY} for $\widetilde g$ and the absolute continuity of integrals yield the desired result.
\end{proof}

Come back to the proof of \cref{thm:RepresentationTheoremForGenerator}. Taking $s=t$ and then the conditional expectation with respect to $\calF_t$ in both sides of BSDE  \eqref{eq:BSDEWithWidetildeYZ} result in the following identity, $\prs$,
  \begin{equation*}
    \frac{1}{\vp}(Y^\vp_t-y)
    =\frac{1}{\vp}\widetilde Y^\vp_t
    =\frac{1}{\vp}\EXlr{\left.\intT[t+\vp]{t}\widetilde g(r,\widetilde Y^\vp_r,\widetilde Z^\vp_r)\dif r\right|\calF_t}.
  \end{equation*}
  Next we set, 
  \begin{align*}
    \Me:=\frac{1}{\vp}\EXlr{\left.
          \intTe \widetilde g(r,\widetilde Y^\vp_r,\widetilde Z^\vp_r)\dif r
          \right|\calF_t},\quad
     \Ne:=\frac{1}{\vp}\EXlr{\left.\intTe \widetilde g(r,0,0)\dif r\right|\calF_t}.
  \end{align*}
  Thus, it holds that
  \begin{equation*}
    \frac{1}{\vp}(Y^\vp_t-y)-g(t,y,z)
    =\frac{1}{\vp}\widetilde Y^\vp_t-\widetilde g(t,0,0)
    =\Me-\Ne+\Ne-\widetilde g(t,0,0).
  \end{equation*}
  Then it is sufficient to prove that $(\Me-\Ne)$ and $(\Ne-\widetilde g(t,0,0))$ tend to $0$ in $L^p$ $(1\leq p<2)$ sense for $\dte$ $t\in[0,T)$ as $\vp\to0^+$, respectively. 
  
  We first treat the term $(\Me-\Ne)$. Since $\widetilde g$ satisfies \ref{H:Integrable}, \ref{H:gContinuousInY} and \ref{H:gGeneralizedGeneralGrowthInY}, we can deduce from \cref{pro:SequenceOfGenerator} and \eqref{eq:TildeYVpBoundedByK} that there exists a process sequence $\{(\widetilde g^n_K(t))_{t\in\tT}\}^\infty_{n=1}$ such that for $\dte$ $t\in\tT$, $\lim_{n\to\infty}\EX[|\widetilde g^n_K(t)|^2]=0$, and $\ptss$, for each $n\geq 1$, \begin{equation}\label{eq:TildeGnLeqPsiKG}
    |\widetilde g^n_K(s)|\leq 2\widetilde\psi_K(t)+4|\widetilde g(s,0,0)|,
  \end{equation}
  \begin{align*}
    |\widetilde g(s,\widetilde Y^\vp_s,\widetilde Z^\vp_s)-\widetilde g(s,0,0)|
    &=|\widetilde g(s,q_K(\widetilde Y^\vp_s),\widetilde Z^\vp_s)-\widetilde g(s,0,0)|
    \leq \lambda|\widetilde Z^\vp_s|+|\widetilde g(s,q_K(\widetilde Y^\vp_s),0)-\widetilde g(s,0,0)|\\
    &\leq \lambda|\widetilde Z^\vp_s|+\widetilde g^n_K(s)+2n|\widetilde Y^\vp_s|.
  \end{align*}
  Then Jesen's and H\"older's inequalities yield that for $\dte$ $t\in[0,T)$, $n\geq 1$ and $1\leq p<2$,
  \begin{align}
    \EXlr{|\Me-\Ne|^p}
    &\leq\EX
    \left[\left(\frac{1}{\vp}\intTe\big|\widetilde g(r,\widetilde Y^\vp_r,\widetilde Z^\vp_r)-\widetilde g(r,0,0)\big|\dif r\right)^p\right]\nonumber\\
    &\leq 2^p\EXlr{\left(\frac{1}{\vp}\intTe\big(2n|\widetilde Y^\vp_r|+\lambda|\widetilde Z^\vp_r|\big)\dif r\right)^p}+2^p\EXlr{\left(\frac{1}{\vp}\intTe|\widetilde g^n_K(r)|\dif r\right)^p}\nonumber\\
    &\leq4^p(2n+\lambda)^p\EXlr{\frac{1}{\vp}\intTe\big(|\widetilde Y^\vp_r|^p+|\widetilde Z^\vp_r|^p\big)\dif r}+2^p\EXlr{\left(\frac{1}{\vp}\intTe|\widetilde g^n_K(r)|\dif r\right)^p}.\label{eq:MEpMinusNEpLeq}
  \end{align}
  The first term on the right hand side of the previous inequality tends to $0$ as $\vp\to0^+$ because of \cref{lem:TidleYZVpLimitZero}. Concerning the second term, we can deduce from \eqref{eq:TildeGnLeqPsiKG} that for each $n\geq1$,
  \[\EX\left[\intT{0}|\widetilde g^n_K(r)|^2\dif r\right]\leq 32\EX\left[\intT{0}\big(|\widetilde \psi_K(r)|^2+|\widetilde g(r,0,0)|^2\big)\dif r\right]<\infty.\]
  Thus, \cref{lem:ExtendLebesgueLemma} indicates that for $\dte$ $t\in[0,T)$, each $n\geq 1$ and $1\leq p<2$,
  \begin{equation*}
    \lim_{\vp\to0^+}\EXlr{\left(\frac{1}{\vp}\intTe \widetilde g^n_K(r)\dif r\right)^p}
    =\EXlr{|\widetilde g^n_K(t)|^p}.
  \end{equation*}
  Note that the right hand side in the previous identity tends to $0$ as $n\to\infty$. By sending $\vp\to0^+$ and then $n\to\infty$ in \eqref{eq:MEpMinusNEpLeq}, we get that for $\dte$ $t\in[0,T)$, $\lim_{\vp\to0^+}\EXlr{|\Me-\Ne|^p}=0$.

  Now we consider the term $(\Ne-\widetilde g(t,0,0))$. Taking into account that \ref{H:Integrable} holds for $\widetilde g$, we can derive from Jensen's inequality and \cref{lem:ExtendLebesgueLemma} that for $\dte$ $t\in[0,T)$ and each $1\leq p<2$, as $\vp\to0^+$,
  \begin{equation*}
    \EXlr{|\Ne-\widetilde g(t,0,0)|^p}
    \leq \EXlr{\left(\frac{1}{\vp}\intTe|\widetilde g(r,0,0)-\widetilde g(t,0,0)|\dif r\right)^p}\to0.
  \end{equation*}
  The proof of \cref{thm:RepresentationTheoremForGenerator} is complete.
\end{proof}

\begin{rmk}
	With analogous techniques and arguments in the proof of \cref{thm:RepresentationTheoremForGenerator}, the representation theorem via coupled forward and backward stochastic differential equations (like \citet*{Jiang2005cSPA}) and in stochastic process spaces (like \citet*{FanJiangXu2011EJP}) can also be proved. 
\end{rmk}

With the help of  \cref{thm:RepresentationTheoremForGenerator}, we can obtain a corresponding converse comparison theorem for solutions of one dimensional BSDEs with weak monotonicity and general growth generators in $y$. Its proof is omitted since it is classical. Detailed arguments are referred to \citet*{Jiang2006ScienceInChina}. 

\begin{thm}[Converse comparison  theorem]\label{thm:SolutionConverseComparisonTheorem}
	Let the generators $g_i$ $(i=1,2)$ satisfy \ref{H:Integrable} -- \ref{H:gLipschitzInZ}. If for each $t\in\tT$ and $\xi\in L^2(\calF_t;\rtn)$, the solutions $\big(Y_s(g_i,t,\xi),Z_s(g_i,t,\xi)\big)_{s\in[0,t]}$ of BSDE $(g_i,t,\xi)$ satisfy that $\prs$, $Y_s(g_1,t,\xi)\geq Y_s(g_2,t,\xi)$ for each $s\in[0,t]$, then for each $y\in\rtn$ and $z\in\rtn^d$,
	\begin{equation*}
	g_1(t,y,z)\geq g_2(t,y,z),\quad\pts.
	\end{equation*}
\end{thm}

\section{Applications to viscosity soultions of  semilinear parabolic PDEs}
\label{sec:Applications}
This section will demonstrate a new application of the representation theorem for generators of BSDEs. Indeed, utilizing the representation theorem obtained in \cref{thm:RepresentationTheoremForGenerator}, we  can construct a probabilistic formula for viscosity solutions of a semilinear parabolic PDE of second order. 

Let $b(t,\tilde x):\tT\times\rtn^n\mapsto\rtn^n$, $\sigma(t,\tilde x):\tT\times\rtn^n\mapsto\rtn^{n\times d}$ be two measurable continuous functions which are Lipschitz continuous in $\tilde x$ uniformly with respect to $t$. Then for each given $t\in\tT$ and $x\in\rtn^n$, the following SDE admits a unique solution $(X^{t,x}_s)_{s\in[t,T]}$ in $\calS^2$,
\begin{equation*}
  X^{t,x}_s=x+\intT[s]{t}b(r,X^{t,x}_r)\dif r+\intT[s]{t}\sigma(r,X^{t,x}_r)\dif B_r,\quad s\in[t,T].
\end{equation*}
Now we consider the following BSDE:
\begin{equation*}
  Y^{t,x}_s=\Phi(X^{t,x}_T)+\intT{s}g(r,X^{t,x}_r,Y^{t,x}_r,Z^{t,x}_r)\dif r-\intT{s}\langle Z^{t,x}_r,\dif B_r\rangle,\quad s\in[t,T],
\end{equation*}
where $\Phi:\rtn^n\mapsto\rtn$ and $g:\tT\times\rtn^n\times\rtn\times\rtn^d\mapsto\rtn$ are continuous. Assume that for each $\tilde x\in\rtn^n$, $g(t,\tilde x,y,z)$ satisfies \ref{H:gWeakMonotonicityInY} -- \ref{H:gLipschitzInZ}, and $\Phi(\cdot)$ and $g(t,\cdot,0,0)$ are polynomial growth, i.e., there exist two constants $L>0$ and $p>0$ such that for each $t\in\tT$, $\tilde x\in\rtn^n$, $y\in\rtn$ and $z\in\rtn^d$,
\begin{equation*}
  |\Phi(\tilde x)|+|g(t,\tilde x,0,0)|\leq L(1+|\tilde x|^p).
\end{equation*} 
Note that \ref{H:Integrable} -- \ref{H:gContinuousInY} are fulfilled trivially for $g(\cdot,\widetilde x,\cdot,\cdot)$. Then the previous BSDE admits a unique solution $(Y^{t,x}_s,Z^{t,x}_s)_{s\in[t,T]}\in\calS^2\times\calH^2$. It is obvious that for each $s\in[t,T]$, $Y^{t,x}_s$ is $\calF^t_s=\sigma\{B_r-B_t,t\leq r\leq s\}\vee\mathcal{N}$ measurable, where $\mathcal{N}$ is the class of the $\PR$-null sets of $\calF$. In particular, $Y^{t,x}_t$ is deterministic. From the uniqueness for solutions of BSDEs, we also have that $Y^{t,x}_{t+\vp}=Y^{t+\vp,X^{t,x}_{t+\vp}}_{t+\vp}$ for any $\vp>0$. 

We define the infinitesimal generator of the Markov process $(X^{t,x}_s)_{s\in[t,T]}$ as follows,
\[
  \calL_s:=\frac{1}{2}\sum_{i,j}(\sigma\sigma^*)(s,\tilde x)\frac{\partial^2}{\partial x_i\partial x_j}+\sum_{i}b_i(s,\tilde x)\frac{\partial}{\partial x_i},\quad s\in\tT,\ \tilde x\in\rtn^n,\ 1\leq i,j\leq n,
\]
where and hereafter $\sigma^*$ represents
the transpose of $\sigma$. Next we will prove that the function $u(t,x):=Y^{t,x}_t$, $(t,x)\in\tT \times\rtn^n$, is a viscosity solution of the following system of backward semilinear parabolic PDE of second order:
\begin{equation}\label{eq:SemilinearParabolicPDEs}
  \begin{cases}
    \displaystyle \partial_t u(t,x)+\calL_tu(t,x)+g(t,x,u(t,x),(\sigma^*\nabla u)(t,x))=0,\quad (t,x)\in[0,T) \times\rtn^n;\\
    \displaystyle u(T,x)=\Phi(x),\quad x\in\rtn^n.
  \end{cases}
\end{equation}

Now we adapt the notion of viscosity solution of PDE \eqref{eq:SemilinearParabolicPDEs} from \citet*{CrandallIshiiLions1992BAMS} and \citet*{Pardoux1999NADEC}.
\begin{dfn}
  Let $u\in C(\tT \times\rtn^n;\rtn)$ satisfy $u(T,x)=\Phi(x)$ for $x\in\rtn^n$. $u$ is called a viscosity subsolution (resp. supersolution) of PDE \eqref{eq:SemilinearParabolicPDEs} if, whenever $\varphi\in C^{1,2}_b(\tT \times\rtn^n;\rtn)$, and $(t,x)\in[0,T)\times\rtn^n$ is a global maximum (resp. minimum) point of $u-\varphi$, we have
  \begin{equation*}
    \partial_t \varphi(t,x)+\calL_t\varphi(t,x)+g(t,x,u(t,x),(\sigma^*\nabla \varphi)(t,x))\geq (\text{resp. }\leq)\  0.
  \end{equation*}
  And $u$ is called a viscosity solution of PDE \eqref{eq:SemilinearParabolicPDEs} if it is both a viscosity subsolution and supersolution.
\end{dfn}
 
\begin{thm}\label{thm:ViscositySoutionsExistence}
  Under the above assumptions, $u(t,x):=Y^{t,x}_t$, $(t,x)\in\tT \times\rtn^n$, is a continuous function of $(t,x)$ which grows at most polynomially and it is a viscosity solution of PDE \eqref{eq:SemilinearParabolicPDEs}.
\end{thm}
\begin{proof}
  The continuity is a consequence of the continuity of $(X^{t,x}_s)_{s\in[t,T]}$ with respect to $(t, x)$ and \cref{lem:EstimateForSolutions}. The polynomial growth follows from classical moment estimates for $(X^{t,x}_s)_{s\in[t,T]}$ and the assumptions on the growth of $\Phi$ and $g$. Note that $u(T,x)=\Phi(x)$ is trivially satisfied. We only prove that $u$ is a viscosity subsolution since the proof of the other assertion is symmetric. 
  
  Take any $\varphi\in C^{1,2}_b(\tT \times\rtn^n;\rtn)$ and $(t,x)\in[0,T)\times\rtn^n$ such that $u-\varphi$ achieves a global maximum at $(t,x)$. Without loss of generality, we set $u(t,x)=\varphi(t,x)$. For $0\leq t\leq t+\vp\leq T$ with $\vp$ small enough and $x\in\rtn^n$, \cref{pro:ExistenceUniquenessOfSolutions} shows that 
  \begin{equation*}
    u(t,x)=u(t+\vp,X^{t,x}_{t+\vp})+\intTe g(r,X^{t,x}_r,Y^{t,x}_r,Z^{t,x}_r)\dif r-\intTe \langle Z^{t,x}_r,\dif B_r\rangle.
  \end{equation*}  
  Let $(\overline Y^\vp_s,\overline Z^\vp_s)_{s\in[t,t+\vp]}\in\calS^2\times\calH^2$ denote the unique solution of the following BSDE:
  \begin{equation*}
    \overline Y^\vp_s=\varphi(t+\vp,X^{t,x}_{t+\vp})+\intTe[s]g(r,X^{t,x}_r,\overline Y^\vp_r,\overline Z^\vp_r)\dif r-\intTe[s]\langle\overline Z^\vp_r,\dif B_r\rangle, \quad s\in[t,t+\vp].
  \end{equation*}
  Then comparison theorem (Theorem 3 in \citet*{Fan2015JMAA}) provides that $u(t,x)=\varphi(t,x)\leq\overline Y^\vp_t$. Moreover, It\^o's formula to $\varphi(r,X^{t,x}_r)$ arrives at
  \begin{equation*}
    \varphi(s,X^{t,x}_s)=\varphi(t+\vp,X^{t,x}_{t+\vp}) -\!\intTe[s]\!(\partial_r\varphi+\calL_r\varphi)(r,X^{t,x}_r)\dif r -\!\intTe[s]\!\langle(\sigma^*\nabla\varphi)(r,X^{t,x}_r),\dif B_r\rangle, \quad\! s\in[t,t+\vp].
  \end{equation*}
  Thus, we obtain that, setting $\hat Y^\vp_\cdot:=\overline Y^\vp_\cdot-\varphi(\;\cdot\;,X^{t,x}_\cdot)$ and $\hat Z^\vp_\cdot:=\overline Z^\vp_\cdot-(\sigma^*\nabla\varphi)(\;\cdot\;,X^{t,x}_\cdot)$, 
  \begin{equation*}
    \hat Y^\vp_s=\intTe[s]G(r,X^{t,x}_r,\hat Y^\vp_r,\hat Z^\vp_r)\dif r-\intTe[s]\langle\hat Z^\vp_r,\dif B_r\rangle,\quad s\in[t,t+\vp],
  \end{equation*}
  where the generator is defined as follows, for each $y\in\rtn$ and $z\in\rtn^d$,
  \begin{equation*}
    G(r,X^{t,x}_r,y,z):=\left(\partial_r\varphi+\calL_r\varphi\right)(r,X^{t,x}_r)+g\big(r,X^{t,x}_r,y+\varphi(r,X^{t,x}_r),z+(\sigma^*\nabla\varphi)(r,X^{t,x}_r)\big).
  \end{equation*}
  Note that $G(t,X^{t,x}_t,y,z)=G(t,x,y,z)$ is deterministic and continuous in $t$.  \cref{thm:RepresentationTheoremForGenerator}, \cref{cor:RepresentationTheoremForGenerator} and \cref{rmk:GContinuousInT} yield that 
  \begin{equation*}
    \left(\partial_t\varphi+\calL_t\varphi\right)(t,x)+g(t,x,u(t,x),(\sigma^*\nabla\varphi)(t,x))=G(t,x,0,0)=\lim_{\vp\to0^+}\frac{1}{\vp}\hat Y^\vp_t,
  \end{equation*}
  which together with the fact $\hat Y^\vp_t\geq 0$ implies the desired result.
\end{proof}

\begin{rmk}
	(i) The required conditions of $g$ in \cref{thm:ViscositySoutionsExistence} are weaker than some existing results including \citet*{Peng1991SSR}, \citet*{PardouxPeng1992SPDETA} and \citet*{Pardoux1999NADEC}. And, the representation theorem approach used in \cref{thm:ViscositySoutionsExistence} overcomes the difficulty that the strict comparison theorem for solution of BSDEs employed in Theorem 3.2 of  \citet*{Pardoux1999NADEC} does not hold in our framework.
	
	(ii) From the proof procedure of \cref{thm:ViscositySoutionsExistence} we can observe that the probabilistic formula for semilinear PDEs holds as soon as the representation theorem for generators holds, where the latter one is determined by the conditions of the original generator $g$. From this point of view, the problem of a probabilistic formula for viscosity solutions of semilinear PDEs can be transformed to a representation problem for the generator of a BSDE.
	
  (iii) Analogous arguments to that in  \cref{thm:ViscositySoutionsExistence} can be extended easily to systems of multidimensional semilinear second order PDEs of both parabolic and elliptic types, more detailed proof procedures are referred to \citet*{Pardoux1999NADEC}. However, for the notion of viscosity solutions to make sense, an extra assumption on $g$ is needed, i.e., the $i$th coordinate of $g$ depends only on the $i$th row of $z$.
\end{rmk}

\end{document}